 \documentclass[authoryear,preprint,review,10pt]{elsarticle}

\usepackage{amssymb}

\newtheorem{theorem}{Theorem}

\newtheorem{definition}[theorem]{Definition}

\newtheorem{lemma}[theorem]{Lemma}

\newtheorem{remark}[theorem]{Remark}

\newenvironment{proof}[1][Proof]{\textbf{#1.} }{\ \rule{0.5em}{0.5em}}

\def \R{\mathbb{R}}

\def \N{\mathbb{N}}
\def \E{\mathbb{E}}

\def \bf{\textbf}

\begin{document}

\begin{frontmatter}



\title{ Time-dependent Neutral stochastic functional differential equation driven by
 a fractional  Brownian motion in a Hilbert space}

\author[a]{Brahim Boufoussi }
\ead{b.boufoussi@uca.ma}
\author[a]{Salah Hajji}
\ead{hajjisalahe@gmail.com}

\author[b]{El Hassan Lakhel}
\ead{e.lakhel@uca.ma}

\address[a]{ Department of Mathematics, Faculty of Sciences Semlalia, Cadi Ayyad University, 2390 Marrakesh, Morocco\\}
\address[b]{  National School of Applied Sciences, Cadi Ayyad University, 46000 Safi , Morocco}
\begin{abstract}
In this paper we consider a class of time-dependent neutral
stochastic functional differential equations with finite delay
driven   by a fractional Brownian motion   in a Hilbert space. We
prove an existence and uniqueness result  for the mild solution by
means of the Banach fixed point principle. A practical example is
provided to illustrate the viability of the abstract result of
this work.
\end{abstract}


\end{frontmatter}

\section{Introduction}
The stochastic functional differential equations have attracted
much attention because of their practical applications in many
areas such as physics, medicine, biology, finance, population
dynamics, electrical engineering, telecommunication networks, and
other fields. For more details, one can see \cite{da}, and
\cite{ren} and the references therein.

In many areas of science, there has been an increasing interest in the investigation of the systems incorporating memory or aftereffect, i.e., there is the
effect of delay on state equations. Therefore, there is a real
need to discuss stochastic evolution systems with delay. In many
mathematical  models the claims often display long-range memories,
possibly due to extreme weather, natural disasters, in some cases,
many stochastic dynamical systems depend not only on present and
past states, but also contain the derivatives with delays. In such
cases, the class of neutral stochastic differential equations
driven by fractional Brownian motion  provide an important tool
for describing and analyzing such systems.  Very recently, neutral
stochastic functional differential  equations driven by fractional
Brownian motion   have attracted the interest of many researchers.
One can see \cite{boufoussi3}, \cite{carab}, \cite{hlak1} and the
references therein.\\ Motivated by the above works, this paper is
concerned with  the existence and uniqueness    of mild solutions
for a class of time-dependent
 neutral functional stochastic differential equations  described in the form:
{\small \begin{equation}\label{eq1}
 \left\{\begin{array}{lll}
d[x(t)+g(t,x(t-r(t)))]=[A(t)x(t)+f(t,x(t-\rho(t))]dt+\sigma
(t)dB^H(t)   ,\;0\leq t \leq T,\nonumber\\
x(t)=\varphi(t) ,\;-\tau \leq t \leq 0,
\end{array}\right.
\end{equation}}
where   \{ $A(t),\,t\in[0,T]\}$ is a family of linear closed
operators from a  Hilbert space $X$ into $X$ that generates an
  evolution system of operators $\{U(t,s),\, 0\leq s\leq t \leq T\}$. $B^H$ is a fractional Brownian motion on a
real and separable Hilbert space $Y$, $r,\; \rho
\;:[0,+\infty)\rightarrow [0,\tau]\; (\tau
>0)$ are continuous and $f,g:[0,+\infty)\times X \rightarrow X,\;
\; \sigma:[0,+\infty) \rightarrow \mathcal{L}_2^0(Y,X)$,$\; \;$
   are appropriate
functions. Here $\mathcal{L}_2^0(Y,X)$ denotes the space of all
$Q$-Hilbert-Schmidt operators from $Y$ into $X$ (see section 2 below).\\

On the other hand, to the best of our knowledge, there is no paper
which investigates the study of time-dependent neutral stochastic
functional differential equations with delays driven   by
fractional Brownian motion. Thus, we will make the first attempt
to study such problem in this paper.\\
 Our results are inspired by the one in
\cite{boufoussi3} where the existence and uniqueness of mild
solutions to  model (\ref{eq1}) with $A(t)=A, \forall t\in [0,T],$
is studied, as well as some results on the  asymptotic behavior.\\

The  substance  of the paper is organized as follows.  Section 2,
recapitulate some notations,  basic  concepts, and basic results
about fractional Brownian motion,  Wiener integral over Hilbert
spaces and we recall some preliminary results about evolution
operator. We need to  prove a new technical lemma for the
$\mathbb{L}^2-$estimate of  stochastic convolution integral which is
different from that used by \cite{boufoussi3}. Section 3, gives
sufficient conditions to prove the existence and uniqueness for
the problem $(\ref{eq1})$. In Section 4 we give an example to illustrate the efficiency of the obtained result.
\section{Preliminaries}\label{sec:1}
\subsection{Evolution families}
In this subsection we introduce the notion of evolution family.

\begin{definition}\label{d1}  A set  $\{U(t, s):  0\leq s \leq t\leq T\}$
of bounded linear operators on a Hilbert space $X$ is called an
\emph{evolution family} if
\begin{itemize}
\item[(a)] $U(t,s)U(s,r)=U(t,r)$, $U(s, s)=I$ if $ r \leq s \leq  t
$,
\item[(b)] $(t,s)\to U(t,s)x$ is strongly continuous for $t> s$.
\end{itemize}
\end{definition}
Let $\{A(t),\, t\in[0,T]\}$ be a family
of closed densely defined linear unbounded operators on the
Hilbert space $X$ and with domain $D(A(t))$ independent of $t$, satisfying the following conditions introduced by   \cite{AT}.

There exist constants $\lambda_0\geq 0$, $\theta\in
(\frac{\pi}{2}, \pi)$, $L$,
 $K\geq 0$, and $\mu$, $\nu\in (0, 1]$ with $\mu +\nu >1$ such that
\begin{equation}\label{AT1}
\Sigma_\theta\cup \{0\}\subset\rho (A(t)-\lambda_0), \quad
\|R(\lambda, A(t)-\lambda_0)\|\leq\frac{K}{1+|\lambda|}
\end{equation}
and
\begin{equation}\label{AT2}
\|(A(t)-\lambda_0)R(\lambda, A(t)-\lambda_0)\big[R(\lambda_0,
A(t)) -R(\lambda_0, A(s))\big]\|\leq L|t-s|^\mu|\lambda|^{-\nu},
\end{equation}
for $t$, $s\in\mathbb{R}$,
$\lambda\in\Sigma_\theta$ where $\Sigma_\theta:=\big\{\lambda\in{\mathbb{C}}-\{0\}:
 |\arg \lambda|\leq\theta\big\}$.

It is well known, that this assumption implies that there exists a unique evolution
family $\{U(t, s):  0\leq s \leq t\leq T\}$ on $X$ such that $(t,
s)\to U(t, s)\in{\mathcal L}(X)$ is continuous for $t>s$,
$U(\cdot, s)\in \mathcal{C}^1((s, \infty), {\mathcal L}(X))$,
 $\partial_tU(t, s)=A(t)U(t, s)$, and
\begin{equation}\label{w1}
\|A(t)^kU(t, s)\|\leq C(t-s)^{-k}
\end{equation}
for $0<t-s\leq 1$, $k=0, 1$, $0\leq \alpha <\mu$, $x\in
D((\lambda_0-A(s))^\alpha)$, and a constant $C$ depending only on
the constants in (\ref{AT1})-(\ref{AT2}). Moreover,
$\partial_s^+U(t, s)x=-U(t, s)A(s)x$ for $t>s$ and $x\in D(A(s))$
with $A(s)x\in \overline {D(A(s))}$. We say that $A(\cdot)$
generates $\{U(t, s): 0\leq s \leq t\leq T\}$. Note that $U(t,s)$
is exponentially bounded by (\ref{w1}) with $k=0$.

\begin{remark} If $\{A(t), \,t\in[0,T]\}$ is a second order
differential operator $A$, that is $A(t)=A$ for each $t\in[0,T]$,
then $A$ generates a $C_0-$semigroup $\{e^{At}, t\in[0,T]\}$.
\end{remark}

For additional details on evolution system and their properties,
we refer the reader to \cite{pazy}.

\subsection{Fractional Brownian Motion}

For the convenience for the reader we recall briefly here some of
the basic results of fractional Brownian motion calculus. For
details of this section, we refer the reader to \cite{nualart} and
the references therein.\\

Let $(\Omega,\mathcal{F}, \mathbb{P})$ be a complete probability
space. A standard fractional Brownian motion (fBm) $\{\beta^H(t),
t\in\mathbb{R}\}$
 with Hurst parameter $H\in (0, 1)$ is a zero mean  Gaussian process with continuous
sample paths such that
\begin{eqnarray}\label{A3}
\mathbb{E}[\beta^H(t)\beta^H(s)]=\frac{1}{2}\big(t^{2H}+s^{2H}-|t-s|^{2H}\big)
\end{eqnarray}
for $s$, $t\in\mathbb{R}$. It is clear that for $H=1/2$, this
process is a standard Brownian motion. In this paper, it is
assumed that $H\in (\frac{1}{2}, 1)$.

This process was introduced by \cite{kol} and later
studied by \cite{MV}.  Its self-similar
and long-range dependence make this process a useful
driving noise in
 models arising in physics, telecommunication networks, finance and other fields.


 Consider a time interval $[0,T]$ with arbitrary fixed
horizon $T$ and let $\{\beta^H(t) , t \in [0, T ]\}$ the
one-dimensional fractional Brownian motion with Hurst parameter
$H\in(1/2,1)$. It is well known that  $\beta^H$ has the following
Wiener integral representation:
\begin{equation}\label{rep}
\beta^H(t) =\int_0^tK_H(t,s)d\beta(s),
 \end{equation}
where $\beta = \{\beta(t) :\; t\in [0,T]\}$ is a Wiener process,
and $K_H(t; s)$ is the kernel given by
$$K_H(t, s )=c_Hs^{\frac{1}{2}-H}\int_s^t (u-s)^{H-\frac{3}{2}}u^{H-\frac{1}{2}}du,
$$
for $t>s$, where $c_H=\sqrt{\frac{H(2H-1)}{\beta (2-2H,H-\frac{1}{2})}}$ and $\beta(,)$
 denotes the Beta function. We put $K_H(t, s ) =0$ if $t\leq s$.\\
We will denote by $\mathcal{H}$ the reproducing kernel Hilbert
space of the fBm. In fact $\mathcal{H}$ is the closure of the set of
indicator functions $\{1_{[0;t]},  t\in[0,T]\}$ with respect to
the scalar product
$$\langle 1_{[0,t]},1_{[0,s]}\rangle _{\mathcal{H}}=R_H(t , s).$$
The mapping $1_{[0,t]}\rightarrow \beta^H(t)$
 can be extended to an isometry between $\mathcal{H}$
and the first  Wiener chaos and we will denote by
$\beta^H(\varphi)$ the image of $\varphi$ by the previous
isometry.

We recall that for $\psi,\varphi \in \mathcal{H}$ their scalar
product in $\mathcal{H}$ is given by
$$\langle \psi,\varphi\rangle _{\mathcal{H}}=H(2H-1)\int_0^T\int_0^T\psi(s)\varphi(t)|t-s|^{2H-2}dsdt.
$$
Let us consider the operator $K_H^*$ from $\mathcal{H}$ to
$\mathbb{L}^2([0,T])$ defined by
$$(K_H^*\varphi)(s)=\int_s^T\varphi(r)\frac{\partial K}{\partial
r}(r,s)dr.
$$ We refer to \cite{nualart} for the proof of the fact
that $K_H^*$ is an isometry between $\mathcal{H}$ and
$L^2([0,T])$. Moreover for any $\varphi \in \mathcal{H}$, we have
$$\beta^H(\varphi)=\int_0^T(K_H^*\varphi)(t)d\beta(t).$$
It follows from \cite{nualart} that the elements of $\mathcal{H}$
may be not functions but distributions of negative order. In the
case $H>\frac{1}{2}$, the second partial derivative of the
covariance function
$$
\frac{\partial R_H}{\partial t\partial s}=\alpha_H|t-s|^{2H-2},
$$
where $\alpha_H=H(2H-2)$, is integrable, and we can write
\begin{equation}\label{r(t,t)}
R_H(t,s)=\alpha_H\int_0^t\int_0^s|u-v|^{2H-2}dudv.
\end{equation}

In order to obtain a space of functions contained in
$\mathcal{H}$, we consider the linear space $|\mathcal{H}|$
generated by the measurable functions $\psi$ such that
$$\|\psi \|^2_{|\mathcal{H}|}:= \alpha_H  \int_0^T \int_0^T|\psi(s)||\psi(t)| |s-t|^{2H-2}dsdt<\infty,
$$
where $\alpha_H = H(2H-1)$. The space $|\mathcal{H}|$ is a Banach
space with the norm  $\|\psi\|_{|\mathcal{H}|}$ and we have the
following inclusions (see \cite{nualart}).
\begin{lemma}\label{lem1}
$$\mathbb{L}^2([0,T])\subseteq \mathbb{L}^{1/H}([0,T])\subseteq |\mathcal{H}|\subseteq \mathcal{H},
$$
and for any $\varphi\in \mathbb{L}^2([0,T])$, we have
$$\|\psi\|^2_{|\mathcal{H}|}\leq 2HT^{2H-1}\int_0^T
|\psi(s)|^2ds.
$$
\end{lemma}
Let $X$ and $Y$ be two real, separable Hilbert spaces and let
$\mathcal{L}(Y,X)$ be the space of bounded linear operator from
$Y$ to $X$. For the sake of convenience, we shall use the same
notation to denote the norms in $X,Y$ and $\mathcal{L}(Y,X)$. Let
$Q\in \mathcal{L}(Y,Y)$ be an operator defined by $Qe_n=\lambda_n
e_n$ with finite trace
 $trQ=\sum_{n=1}^{\infty}\lambda_n<\infty$. where $\lambda_n \geq 0 \; (n=1,2...)$ are non-negative
  real numbers and $\{e_n\}\;(n=1,2...)$ is a complete orthonormal basis in $Y$.
 Let $B^H=(B^H(t))$ be  $Y-$ valued fbm on
  $(\Omega,\mathcal{F}, \mathbb{P})$ with covariance $Q$ as
 $$B^H(t)=B^H_Q(t)=\sum_{n=1}^{\infty}\sqrt{\lambda_n}e_n\beta_n^H(t),
 $$
 where $\beta_n^H$ are real, independent fBm's. This process is  Gaussian, it
 starts from $0$, has zero mean and covariance:
 $$E\langle B^H(t),x\rangle\langle B^H(s),y\rangle=R(s,t)\langle Q(x),y\rangle \;\; \mbox{for all}\; x,y \in Y \;\mbox {and}\;  t,s \in [0,T]$$
In order to define Wiener integrals with respect to the $Q$-fBm,
we introduce the space $\mathcal{L}_2^0:=\mathcal{L}_2^0(Y,X)$  of
all $Q$-Hilbert-Schmidt operators $\psi:Y\rightarrow X$. We recall
that $\psi \in \mathcal{L}(Y,X)$ is called a $Q$-Hilbert-Schmidt
operator, if
$$  \|\psi\|_{\mathcal{L}_2^0}^2:=\sum_{n=1}^{\infty}\|\sqrt{\lambda_n}\psi e_n\|^2 <\infty,
$$
and that the space $\mathcal{L}_2^0$ equipped with the inner
product
$\langle \varphi,\psi \rangle_{\mathcal{L}_2^0}=\sum_{n=1}^{\infty}\langle
\varphi e_n,\psi e_n\rangle$ is a separable Hilbert space.\\

Now, let $\phi(s);\,s\in [0,T]$ be a function with values in
$\mathcal{L}_2^0(Y,X)$, such that $\sum_{n=1}^{\infty}\|K^*\phi
Q^{\frac{1}{2}}e_n\|_{\mathcal{L}_2^0}^2<\infty.$ The Wiener
integral of $\phi$ with respect to $B^H$ is defined by

\begin{equation}\label{int}
\int_0^t\phi(s)dB^H(s)=\sum_{n=1}^{\infty}\int_0^t
\sqrt{\lambda_n}\phi(s)e_nd\beta^H_n(s)=\sum_{n=1}^{\infty}\int_0^t
\sqrt{\lambda_n}(K_H^*(\phi e_n)(s)d\beta_n(s)
\end{equation}
where $\beta_n$ is the standard Brownian motion used to  present $\beta_n^H$ as in $(\ref{rep})$.\\
Now, we end this subsection by stating the following result which
is fundamental to prove our result. It can be proved by  similar
arguments as those used to prove   Lemma 2 in \cite{carab}.
\begin{lemma}\label{lem2}
If $\psi:[0,T]\rightarrow \mathcal{L}_2^0(Y,X)$ satisfies
$\int_0^T \|\psi(s)\|^2_{\mathcal{L}_2^0}ds<\infty$,
 then the above sum in $(\ref{int})$ is well defined as a $X$-valued random variable and
 we have$$ \mathbb{E}\|\int_0^t\psi(s)dB^H(s)\|^2\leq 2Ht^{2H-1}\int_0^t \|\psi(s)\|_{\mathcal{L}_2^0}^2ds.
 $$
\end{lemma}

\subsection{The stochastic convolution integral}
 In this subsection we
present a few properties of the stochastic convolution integral of the form
$$
Z(t)=\int_{0}^{t}U(t,s)\sigma(s)dB^H(s),\qquad t\in[0,T],
$$
where $\sigma (s)\in \mathcal{L}_2^0(Y,X)$ and $\{U(t,s),\, 0\leq s\leq t \leq T\}$ is  an
evolution system of operators. \\

The properties of the process $Z$ are crucial when regularity of
the mild solution to stochastic evolution equation is studied, see \cite{da} for asystematic account of the theory of mild solutions
to infinite-dimensional stochastic equations. Unfortunately, the
process $Z$ is not a martingale, and standard tools of the
martingale theory, yielding e.g. continuity of the trajectories or
$\mathbb{L}^2-$estimates are not available.

 The following result on the stochastic convolution integral $Z$ holds.

\begin{lemma}\label{lem3}

Suppose that $\sigma:[0,T]\rightarrow \mathcal{L}_2^0(Y,X)$ satisfies
${\sup_{t\in[0,T]} \|\sigma(t)\|^2_{\mathcal{L}_2^0}<\infty}$,and Suppose that $\{U(t,s),\, 0\leq s\leq t \leq
T\}$ is an evolution system of operators satisfying  $
\|U(t,s)\|\leq Me^{-\beta(t-s)},$  for some constants $\beta>0$
and $M\geq 1$   for  all   $t\geq s. $ Then, we have $$ \mathbb{E}\|\int_0^tU(t,s)\sigma(s)dB^H(s)\|^2\leq C  M^2t^{2H}
  (\sup_{t\in[0,T]} \|\sigma(t)\|_{\mathcal{L}_2^0})^2.$$
\end{lemma}

\begin{proof} Let $\{e_n\}_{n\in\N}$ be the complete orthonormal
basis of $Y$ and $\{\beta_n^H\}_{n\in\N}$ is a sequence of
independent, real-valued standard fractional Brownian motion each
with the same Hurst parameter $H\in(\frac{1}{2},1)$. Thus, using
fractional It\^o isometry one can write

{\small
 \begin{eqnarray*}
   \mathbb{E}\|\int_0^tU(t,s)\sigma(s)dB^H(s)\|^2&=& \sum_{n=1}^{\infty} \mathbb{E}\|\int_0^tU(t,s)\sigma(s)e_nd\beta_n^H(s)\|^2\\
  &=&\sum_{n=1}^{\infty}\int_0^t\int_0^t<U(t,s)\sigma(s)e_n,U(t,r)\sigma(r)e_n>\\
 &\times&  H(2H-1)|s-r|^{2H-2}dsdr\\
&\leq& H(2H-1)\int_0^t\{\|U(t,s)\sigma(s)\|\\
&\times&\int_0^t\|U(t,r)\sigma(r)\||s-r|^{2H-2}dr\}ds\\
&\leq& H(2H-1)M^2\int_0^t\{e^{-\beta(t-s)}\|\sigma(s)\|_{\mathcal{L}_2^0}\\
&\times&\int_0^te^{-\beta(t-r)}|s-r|^{2H-2}\|\sigma(r)\|_{\mathcal{L}_2^0}dr\}ds.\\
 \end{eqnarray*}
 }
Since $\sigma$ is bounded, one can then conclude that {\small
 \begin{eqnarray*}
   \mathbb{E}\|\int_0^tU(t,s)\sigma(s)dB^H(s)\|^2&\leq &H(2H-1)M^2(
   \sup_{t\in[0,T]}
    \|\sigma(t)\|_{\mathcal{L}_2^0})^2\int_0^t\{e^{-\beta(t-s)} \\
   &\times&\int_0^te^{-\beta(t-r)}|s-r|^{2H-2}dr\}ds.
 \end{eqnarray*}
 }
 Make the following change of variables, $v=t-s$ for the first integral and
 $u=t-r$  for the second. One can write
{\small
 \begin{eqnarray*}
   \mathbb{E}\|\int_0^tU(t,s)\sigma(s)dB^H(s)\|^2&\leq &H(2H-1)M^2(
   \sup_{t\in[0,T]}
    \|\sigma(t)\|_{\mathcal{L}_2^0})^2\int_0^t\{e^{-\beta v} \\
   &\times&\int_0^te^{-\beta u}|u-v|^{2H-2}du\}dv\\
   &\leq&H(2H-1)M^2 (
   \sup_{t\in[0,T]}    \|\sigma(t)\|_{\mathcal{L}_2^0})^2\int_0^t\int_0^t|u-v|^{2H-2}dudv.\\
 \end{eqnarray*}
 }
By using (\ref{r(t,t)}), we get that
$$ \mathbb{E}\|\int_0^tU(t,s)\sigma(s)dB^H(s)\|^2\leq C  M^2t^{2H}
  (\sup_{t\in[0,T]} \|\sigma(t)\|_{\mathcal{L}_2^0})^2.
 $$
\end{proof}

\begin{remark}
Thanks to Lemma \ref{lem3}, the stochastic integral $Z(t)$ is
well-defined.
\end{remark}
\section{Existence and Uniqueness of Mild Solutions}

In this section we study the existence and uniqueness of mild
 solutions of equation (\ref{eq1}).  Henceforth we will assume that the family $\{A(t),\,t\in[0,T]\}$
 of linear operators generates
an evolution system of operators $\{U(t,s),\, 0\leq s\leq t \leq
T\}$. \\
Before stating and proving the main result, we give the definition
of mild solutions for equation (\ref{eq1}).

\begin{definition}
A $X$-valued  process $\{x(t),\;t\in[-r,T]\}$, is called  a mild
solution of equation (\ref{eq1}) if
\begin{itemize}
\item[$i)$] $x(.)\in \mathcal{C}([-r,T],\mathbb{L}^2(\Omega,X))$,
\item[$ii)$] $x(t)=\varphi(t), \, -r \leq t \leq 0$.
\item[$iii)$]For arbitrary $t \in [0,T]$, $x(t)$ satisfies the
following integral equation:
\begin{eqnarray*}
x(t)&=&U(t,0)(\varphi(0)+g(0,\varphi(-\rho(0))))-g(t,x(t-\rho(t)))\\
&-& \int_0^t U(t,s)A(s)g(s,x(s-\rho(s)))ds +\int_0^t U(t,s)f(s,x(s-\rho (s))ds\\
&+&\int_0^tU(t,s)\sigma(s)dB^H(s)\;\;\; \mathbb{P}-a.s
\end{eqnarray*}
\end{itemize}
\end{definition}
We introduce  the following  assumptions:

\begin{itemize}
\item [$(\mathcal{H}.1)$]
\begin{itemize}
  \item [$i)$]The evolution family is exponentially
stable, that is, there exist two constants $\beta>0$ and $M\geq 1$
such that
$$
\|U(t,s)\|\leq Me^{-\beta(t-s)},\qquad     for\, all \quad t\geq
s,
$$
  \item [$ii)$] There exist a constant $M_*>0$\, such that
$$
\|A^{-1}(t)\| \leq M_*\qquad for \; all \quad t\in [0,T].
$$
\end{itemize}

\item [$(\mathcal{H}.2)$] The maps   $f, g:[0,T]\times X \rightarrow X
$ are continuous  functions and there  exist two positive
constants $C_1$ and $C_2$, such that for all $t\in[0,T]$ and
$x,y\in X$:
\begin{itemize}
  \item[$i)$] $\|f(t,x)-f(t,y)\|\vee \|g(t,x)-g(t,y)\|\leq C_1\|x-y\|$.
  \item[$ii)$]  $\|f(t,x)\|^2 \vee \|A^k(t)g(t,x)\|^2\leq C_2(1+\|x\|^2),\quad k=0,1.$

\end{itemize}

\item [$(\mathcal{H}.3)$]

\begin{itemize}
  \item [$i)$] There exists a constant $0<L_*<\frac{1}{M_*}$ such that
$$
\|A(t)g(t,x)-A(t)g(t,y)\|\leq L_*\|x-y\|,
$$
for all $t\in[0,T]$ and $x,y\in X$.
  \item [$ii)$] The function $g$ is continuous in the quadratic mean
  sense: for all $x(.)\in\mathcal{C}([0,T],L^2(\Omega,X))$,
  we have
  $$
\lim_{t\longrightarrow s}\E\|g(t(x(t))-g(s,x(s))\|=0.
  $$
\end{itemize}

\item [$(\mathcal{H}.4)$]

\begin{itemize}
  \item [$i)$]The map  $\sigma:[0,T]\longrightarrow \mathcal{L}^0_2(Y,X)$
is bounded, that is :  there exists a positive
constant such that $\|\sigma (t)\|_{\mathcal{L}^0_2(Y,X)}\leq L$
uniformly in $t \in [0,T]$.\\
  \item [$ii)$]The maps  $r,\rho:[0,\infty]\rightarrow \mathbb{R}$
  are
a continuous
  functions  satisfying the condition that
$$
 -\tau\leq\rho(t),r(t)\leq t, \; \forall t\geq0.
$$
\end{itemize}
Moreover, we assume that $\varphi \in \mathcal{C}([-\tau,0],\mathbb{L}^2(\Omega,X))$.\\
\end{itemize}

The main result of this paper is given in the next theorem.
\begin{theorem}\label{jthm1}
Suppose that $(\mathcal{H}.1)$-$(\mathcal{H}.4)$  hold. Then, for
all $T>0$,  the equation (\ref{eq1}) has a unique mild solution on
$[-\tau,T]$.
\end{theorem}

\begin{proof}

Fix $T > 0$ and let $B_T :=
\mathcal{C}([-\tau,T],\mathbb{L}^2(\Omega, X))$ be the Banach
space of all continuous functions from $[-\tau, T]$ into
$\mathbb{L}^2(\Omega, X))$, equipped with the supremum norm
$$\|x\|_{B_T}^2=\sup_{-\tau\leq t \leq T}\mathbb{E}\|x(t,\omega)\|^2.$$
 Let us consider the set
 $$S_T(\varphi)=\{x\in B_T : x(s)=\varphi(s),\; \mbox {for} \;\;s \in [-\tau,0] \}.$$
 $S_T(\varphi)$ is a closed subset of $B_T$ provided with the norm  $\|.\|_{B_T}.$\\
We transform     (\ref{eq1}) into a fixed-point problem.  Consider
the operator
 $\psi$ on $S_T(\varphi)$ defined  by
$\psi(x)(t)=\varphi(t)$ for $t\in [-\tau,0]$
 and for $t\in [0,T]$
 {\small
 \begin{eqnarray*}
   \psi(x) (t)&=&U(t,0)(\varphi(0)+g(0,\varphi(r(0))))-g(t,x(t-r(t) ))\\
  &-&\int_0^t  U(t,s)A(s)g(s,x(s-r(s)))ds
 + \int_0^t U(t,s)f(s,x(s-\rho (s))ds\\
 &+& \int_0^tU(t,s)\sigma(s)dB^H(s)\;\\
&=&  \sum_{i=1}^5I_i(t).
 \end{eqnarray*}
 }
Clearly, the fixed points of the operator $\psi$ are mild
solutions of (\ref{eq1}). The fact that $\psi$ has a fixed point will
be proved in several steps. We will first prove that the function $\psi$ is well defined. \\
 {\bf Step 1:} $\psi$ is well defined.  Let $x \in S_T(\varphi)$ and $t\in [0,T]$, we are going to show that each function $t\rightarrow I_i(t)$
 is continuous  on $[0,T]$ in the $\mathbb{L}^2(\Omega,X)$-sense.\\

From Definition \ref{d1}, we obtain
$$
\lim_{h\longrightarrow0}(U(t+h,0)-U(t,0))(\varphi(0)+g(0,\varphi(r(0))))=0.
$$
From $(\mathcal{H}.1)$, we have
$$
\|(U(t+h,0)-U(t,0))(\varphi(0)+g(0,\varphi(r(0))))\|\leq
Me^{-\beta t}(e^{-\beta h}+1)\|\varphi(0)+g(0,\varphi(r(0)))\|\in
L^2(\Omega).
$$
Then we conclude by the Lebesgue dominated theorem that
$$
\lim_{h\longrightarrow0}\E\|I_1(t+h)-I_1(t)\|^2=0.
$$

 Moreover, assumption $(\mathcal{H}.2)$ ensures that
$$
\lim_{h\longrightarrow0}\E\|I_2(t+h)-I_2(t)\|^2=0.
$$

To show that the third term $I_3(h)$ is continuous, we suppose
$h>0$ (similar calculus for $h<0$). We have

\begin{eqnarray*}
|I_3(t+h)-I_3(t)|&\leq & \left|\int_0^t ( U(t+h,s)-U(t,s))A(s) g(s,x(s-r(s)))ds\right|\\
&& +\left|\int_t^{t+h} (U(t,s)g(s,x(s-r(s)))ds\right|\\
&\leq & I_{31}(h)+I_{32}(h).
\end{eqnarray*}

By H\"older's inequality, we have

$$
\E\|I_{41}(h)\|\leq t\E\int_0^t \|
U(t+h,s)-U(t+h,s))A(s)g(s,x(s-r(s))\|^2ds.
$$
By Definition \ref{d1}, we obtain
$$
\lim_{h\longrightarrow0}( U(t+h,s)-U(t,s))A(s)g(s,x(s-r(s)))=0.
$$
From $(\mathcal{H}.1)$ and $(\mathcal{H}.2)$, we have
$$
\|U(t+h,s)-U(t,s))A(s)g(s,x(s-r(s))\|\leq
C_2Me^{-\beta(t-s)}(e^{-\beta h }+1) \|A(s)g(s,x(s-r(s))\|\in
L^2(\Omega).
$$
 Then we conclude by the Lebesgue dominated theorem that
$$
\lim_{h\longrightarrow0}\E\|I_{31}(h)\|^2=0.
$$
So, estimating as before. By  using $(\mathcal{H}.1)$ and
$(\mathcal{H}.2)$, we get
$$
\E\|I_{32}(h)\|^2\leq \frac{M^2C_2(1-e^{-2\beta
h})}{2\beta}\int_t^{t+h}(1+\E\|x(s-r(s))\|^2ds.
$$
Thus,
$$
\lim_{h\longrightarrow0}\E\|I_{32}(h)\|^2=0.
$$

 For the fourth term $I_4(h)$, we suppose $h>0$
(similar calculus for $h<0$). We have
\begin{eqnarray*}
|I_4(t+h)-I_4(t)|&\leq & \left|\int_0^t ( U(t+h,s)-U(t,s))f(s,x(s-\rho(s)))ds\right|\\
&& +\left|\int_t^{t+h} (U(t,s)f(s,x(s-\rho(s)))ds\right|\\
&\leq & I_{41}(h)+I_{42}(h).
\end{eqnarray*}

By H\"older's inequality, we have
$$
\E\|I_{41}(h)\|\leq t\E\int_0^t \|
U(t+h,s)-U(t+h,s))f(s,x(s-\rho(s))\|^2ds.
$$

Again exploiting properties of Definition \ref{d1}, we obtain
$$
\lim_{h\longrightarrow0}( U(t+h,s)-U(t,s))f(s,x(s-\rho(s)))=0,
$$
and
$$
\|U(t+h,s)-U(t,s))f(s,x(s-\rho(s))\|\leq
Me^{-\beta(t-s)}(e^{-\beta h }+1)\|f(s,x(s-\rho(s))\|\in
L^2(\Omega).
$$
 Then we conclude by the Lebesgue dominated theorem that
$$
\lim_{h\longrightarrow0}\E\|I_{41}(h)\|^2=0.
$$

 On the  other hand, by $(\mathcal{H}.1)$ ,  $(\mathcal{H}.2)$, and the H\"older's inequality, we have
\begin{eqnarray*}
\E\|I_{42}(h)\|\leq\frac{M^2C_2(1-e^{-2\beta
h})}{2\beta}\int_t^{t+h} (1+\E\|x(s-\rho(s)\|^2)ds.
\end{eqnarray*}
Thus
 $$
 \lim_{h\rightarrow 0}I_{42}(h)=0.$$

Now, for the term $I_5(h)$, we have
\begin{eqnarray*}
I_5(h)&\leq & \|\int_0^t (U(t+h,s)-U(t,s)\sigma(s)dB^H(s)\|\\
&+&\|\int_t^{t+h} U(t+h,s)\sigma(s)dB^H(s)\|\\
&\leq & I_{51}(h)+I_{52}(h).
\end{eqnarray*}

By   Lemma \ref{lem2}, we get that
\begin{eqnarray*}
E|I_{51}(h)|^2&\leq &2Ht^{2H-1}\int_0^t \|[U(t+h,s)-U(t,s)]\sigma(s)\|_{\mathcal{L}_2^0}^2ds.\\
\end{eqnarray*}
 Since
 $$
 \displaystyle\lim_{h\rightarrow 0} \|[U(t+h,s)-U(t,s)] \sigma(s)\|_{\mathcal{L}_2^0}^2=0
 $$
  and
 $$\|(U(t+h,s)-U(t,s) \sigma(s)\|_{\mathcal{L}_2^0}\leq (MLe^{-\beta (t-s)}(e^{-\beta h+1})\in \mathbb{L}^1([0,T],\, ds),$$
 we conclude, by the dominated convergence theorem that,
 $$ \lim_{h\rightarrow 0}\mathbb{E}|I_{51}(h)|^2=0.  $$
 Again by Lemma \ref{lem2}, we get that

$$
\mathbb{E}|I_{52}(h)|^2\leq  \frac{2Ht^{2H-1}LM^2(1-e^{-2\beta h
})}{2\beta}.
$$
Thus,
 $$ \lim_{h\rightarrow 0}\mathbb{E}|I_{52}(h)|^2=0.  $$
 The above arguments show that $\displaystyle\lim_{h\rightarrow
0}\mathbb{E}\|\psi(x)(t+h)-\psi(x)(t)\|^2=0$.  Hence, we conclude
that  the function  $t \rightarrow \psi(x)(t)$ is continuous on
$[0,T]$ in the $\mathbb{L}^2$-sense.

{\bf Step 2:}    Now, we are going to show that $\psi$ is a
contraction mapping in $S_{T_1}(\varphi)$ with some $T_1\leq T$ to
be specified later.
 Let $x,y\in S_T(\varphi)$,  by using the inequality $(a+b+c)^2\leq \frac{1}{\nu}a^2+\frac{2}{1-\nu}b^2+\frac{2}{1-\nu}c^2,$
  where $\nu:=L_*M_*<1,$ we obtain for any fixed  $t\in [0,T]$
\begin{eqnarray*}
\|\psi(x)(t)&-&\psi(y)(t)\|^2\\
&\leq&\frac{1}{\nu}\|g(t,x(t-r(t)))-g(t,y(t-r(t)))\|^2\\
&&+\frac{2}{1-\nu}\|\int_0^t U(t,s)A(s)(g(s,x(s-r(s)))-g(s,y(s-r(s)))ds\|^2\\
&&+\frac{2}{1-\nu}\|\int_0^tU(t,s)(f(s,x(s-\rho(s)))-f(s,y(s-\rho(s)))ds\|^2\\
&=& \sum_{k=1}^3J_k(t).
\end{eqnarray*}

By using the fact that the operator $\|(A^{-1}(t))\|$ is bounded,
combined with the condition $(\mathcal{H}.3)$,  we obtain that

\begin{eqnarray*}
\E\|J_1(t)\|&\leq&\|A^{-1}(t)\|^2\E|A(t)g(t,x(t-r(t)))-A(t)g(t,y(t-r(t)))\|\\
&\leq& \frac{L_*^2M_*^2}{\nu}\mathbb{E}\|x(t-r(t))-y(t-r(t))\|^2\\
&\leq & \nu
\sup_{s\in[-\tau,t]} \mathbb{E}\|x(s)-y(s)\|^2.\\
\end{eqnarray*}

  By hypothesis $(\mathcal{H}.3)$  combined with  H\"older's inequality, we get that
   \begin{eqnarray*}
\E\|J_2(t)\|&\leq&
\E\|\int_0^t U(t,s)\left[A(t)g(t,x(t-r(t)))-A(t)g(t,y(t-r(t)))\right]ds\|\\
&\leq& \frac{2}{1-\nu}\int_0^t M^2e^{-2\beta(t-s)}ds
\int_0^t\mathbb{E}\|x(s-r)-y(s-r)\|^2ds\\
&\leq & \frac{2M^2L_*^2}{1-\nu}\frac{1-e^{-2\beta t}}{2\beta}t
\sup_{s\in[-\tau,t]} \mathbb{E}\|x(s)-y(s)\|^2.\\
\end{eqnarray*}

Moreover, by hypothesis$(\mathcal{H}.2)$  combined with H\"older's
inequality, we can conclude that
 \begin{eqnarray*}
E\|J_3(t)\|&\leq&
E\|\int_0^t U(t,s)\left[f(s,x(s-\rho(s)))-f(s,y(s-\rho(s)))\right]ds\|\\
&\leq& \frac{2C_1^2}{1-\nu}\int_0^t M^2e^{-2\beta(t-s)}ds
\int_0^t\mathbb{E}\|x(s-r)-y(s-r)\|^2ds\\
&\leq & \frac{2M^2C_1^2}{1-\nu}\frac{1-e^{-2\beta t}}{2\beta}t
\sup_{s\in[-\tau,t]} \mathbb{E}\|x(s)-y(s)\|^2.\\
\end{eqnarray*}

Hence
$$\sup_{s\in[-\tau,t]}\mathbb{E}\|\psi(x)(s)-\psi(y)(s)\|^2\leq
\gamma(t) \sup_{s\in[-\tau,t]} \mathbb{E}\|x(s)-y(s)\|^2,
$$
where

$$
\gamma(t)=\nu+[L_*^2+C_1^2]\frac{2M^2}{1-\nu}\frac{1-e^{-2\beta
t}}{2\beta}t
$$
 By condition    $(\mathcal{H}.3)$, we have
$\gamma(0)=\nu= L_*M_* <1$. Then there exists $0<T_1\leq T $ such
that $0<\gamma(T_1)<1$ and $\psi$ is a contraction mapping on
$S_{T_1}(\varphi)$ and therefore has a unique fixed point, which
is a mild solution of equation (\ref{eq1}) on $[-\tau,T_1]$. This
procedure can be repeated in order to extend the solution to the
entire interval $[-\tau,T]$ in finitely  many steps. This
completes the proof.
\end{proof}

\section{An Example}
Let us consider the following stochastic partial neutral
functional differential equation with finite variable delays
driven by a cylindrical fractional Brownian motion:
\begin{equation}\label{pe}
\left\{\begin{array}{lll}
   d\left[u(t,\zeta)+G(t,u(t-r(t),\zeta))\right]&= &
    \left[\frac{\partial^2}{\partial^2\zeta}u(t,\zeta)+b(t,\zeta)u(t,\zeta)+F(t, u(t-\rho(t),\zeta))\right]dt\\
  &+& \sigma(t) dB^H(t)\\
 u(t,0)+G(t,u(t-r(t),0)) = 0,&& \hspace{-.5cm} t\geq0 \\
 u(t,\pi)+G(t,u(t-r(t),\pi)) = 0,&&\hspace{-.5cm}t\geq0 \\
   u(t,\zeta)= \varphi(t,\zeta),&&\hspace{-2.5cm} t\in[-\tau,0]\; 0\leq \zeta
   \leq\pi,
  \end{array}\right.
\end{equation} where $B^H$ is a fractional Brownian motion,
$b(t,\zeta)$ is a continuous function and is uniformly H\"older
continuous in $t$,  $F$, $G:\R^+\times\R\longrightarrow\R$ are
continuous
functions.\\
To study this system, we consider the space $X=\mathbb{L}^2([0,\pi],\R)$
and   the operator $A:D(A)\subset X\longrightarrow X$ given  by
$Ay=y{''}$ with
$$
D(A)=\{y\in X: y''\in X,\quad y(0))=y(\pi)=0\}.
$$
 It is well known that $A$ is the infinitesimal generator of an
analytic semigroup \{$T(t)\}_{t\geq 0}$ on $X$. Furthermore,
$A$ has discrete spectrum with eigenvalues $-n^2,\, n\in \N$
 and the corresponding normalized eigenfunctions given by
 $$
e_n:=\sqrt{\frac{2}{\pi}}\sin nx,\; n=1,2,....
$$
In addition  $(e_n)_{n\in\N}$ is  a complete orthonormal basis in
$X$ and $$ T(t)x=\sum_{n=1}^{\infty}e^{-n^2t}<x,e_n>e_n
$$
for $x\in X$ and $t\geq0$. It follows from this representation
that $T(t)$ is compact for every $t>0$ and that $\|T(t)\|\leq
e^{-t}$ for every $t\geq0$.\\
  \\
On the domain $D(A)$,   we define the  operators $A(t):D(A)\subset
X\longrightarrow X$ by
$$
A(t)x(\zeta)=Ax(\zeta)+b(t,\zeta)x(\zeta).
$$
By assuming that $b(.,.)$ is continuous and that
$b(t,\zeta)\leq-\gamma$ $(\gamma>0)$ for every $t\in \R$,
$\zeta\in[0,\pi]$,  it follows that the system
$$
  \left\{
\begin{array}{ll}
u'(t)&=  A(t)u(t), \quad t\geq s,   \\
u(s) &=  x\in X , \\
\end{array}
\right.
$$
has an associated evolution family given by
$$
U(t,s)x(\zeta)=\left[T(t-s)\exp(\int_s^t
b(\tau,\zeta))d\tau)x\right](\zeta).
$$
From this expression, it follows that $U(t,s)$ is a compact linear
operator and that for every $s,t\in[0,T]$ with $t>s$
$$
\|U(t,s)\|\leq e^{-(\gamma+1)(t-s)}
$$
In addition, $A(t)$ satisfies the assumption $\mathcal{H}_1- ii)$
(see \cite{baghli}.\\
In order to define the operator $Q:
Y:=L^2([0,\pi],\R)\longrightarrow Y$, we choose a sequence
$\{\lambda_n\}_{n\in\N}\subset \R^+$, set $Qe_n=\lambda_ne_n$, and
assume that
$$
tr(Q)=\sum_{n=1}^\infty \sqrt{\lambda_n}<\infty.
$$
 Define the fractional Brownian motion in $Y$ by
$$
B^H(t)=\sum_{n=1}^\infty \sqrt{\lambda_n}\beta^H(t)e_n,
$$
where $H\in(\frac{1}{2},1)$ and $\{\beta^H_n\}_{n\in\N}$ is a
sequence of one-dimensional fractional Brownian motions mutually
independent.\\
To write the initial-boundary value problem (\ref{pe}) in the
abstract form we assume the following: Letting $u(t)(.)=u(t,.)$.
For $t\in[0,T]$, $u\in X$ and $\zeta\in[0,\pi]$

\begin{itemize}
  \item [$i)$] The substitution operator $f:[0,T]\times X\longrightarrow
  X$ defined by $f(t,u)(.)=F(t,u(.))$ is continuous and we impose
  suitable conditions on F to verify assumption  $\mathcal{H}_2$.
  \item [$ii)$] The substitution operator $g:[0,T]\times X\longrightarrow
  X$ defined by $g(t,u)(.)=G(t,u(.))$ is continuous and we impose
  suitable conditions on G to verify assumptions  $\mathcal{H}_2$ and $\mathcal{H}_3$.
  \item [$iii)$] The function $\sigma:[0,T]\longrightarrow \mathcal{L}^0_2(L^2([0,\pi],\R),L^2([0,\pi],\R))$
   is bounded, that is,  there exists a positive
constant $L$ such that $\|\sigma (t)\|_{\mathcal{L}^0_2}\leq L$
uniformly in $t \in[0,T]$.\\
\end{itemize}

Thus the problem (\ref{pe}) can be written in the abstract form
{\small \begin{eqnarray*}
 \left\{\begin{array}{lll}
d[x(t)+g(t,x(t-r(t)))]=[A(t)x(t)+f(t,x(t-\rho(t))]dt+\sigma
(t)dB^H(t)   ,\;0\leq t \leq T,\nonumber\\
x(t)=\varphi(t) ,\;-\tau \leq t \leq 0.
\end{array}\right.
\end{eqnarray*}
}

Furthermore, if we impose suitable condition on the  delay
functions $r(.)$,  $\rho(.)$ and on the initial value $\varphi$ to
verify assumptions on theorem \ref{jthm1}, we can conclude that
the system \ref{pe} has a unique mild solution on $[-\tau,T]$.




\section*{References}

\begin{thebibliography}{30}

\bibitem [Acquistapace and Terreni (1987)]{AT} { Acquistapace P.,   Terreni
B., 1987. } {A unified approach to abstract linear parabolic
equations, Tend. Sem. Mat. Univ. Padova \textbf{78} 47-107.}


\bibitem [ Aouad and Baghli (2013)]{baghli}
{Aouad D.,  Baghli S.,  2013. }{ Mild solutions for Perturbed
 evolution equations with infinite state-dependent  delay in Fr\'echet spaces. } Electronic Journal of Differential Equations, Vol. 2008 , No. 68; pp.1-19.

\bibitem [Boufoussi and Hajji (2011)]  {boufoussi1}
 {Boufoussi, B.,   Hajji, S., 2011.} {Functional differential
equations driven by a fractional Brownian motion}. Computers and
Mathematics with Applications 62, 746-754

\bibitem[Boufoussi and Hajji (2012)]  {boufoussi3}
{Boufoussi, B., Hajji, S., 2012.}  {Neutral stochastic functional
differential equation driven by a  fractional  Brownian motion in
a Hilbert space}. Statist. Probab. Lett. 82,  1549-1558.


\bibitem [Boufoussi et al.(2012)] {boufoussi2}
{Boufoussi, B.,  Hajji, S.,  Lakhel, E.,  2011.}  {Functional
differential equations in Hilbert spaces driven by a fractional
Brownian motion}. Afrika Matematika, Volume 23, Issue 2, pp
173-194.



\bibitem [Caraballo et al. (2011)]  {carab}
{ Caraballo, T., Garrido-Atienza, M.J.,  Taniguchi, T., 2011.} {The existence and exponential behavior of solutions to stochastic delay evolution equations with a fractional Brownian motion}. Nonlinear Analysis 74, 3671-3684.


\bibitem [Da Prato and Zabczyk (1992)]{da}
{Da Prato, G., Zabczyk, J., 1992.}  {Stochastic Equations in Infinite Dimension. Cambridge University Press, Cambridge.}


\bibitem [Hajji and Lakhel (2013)] {hlak1}
 {Hajji, S.,  Lakhel, E., 2013.}  {Existence and uniqueness of mild solutions to
  neutral Sfdes driven by a fractional Brownian motion with non-Lipschitz coefficients}.    Preprint: arXiv:1312.6147 [math.DS].

\bibitem [Kolmogorov (1940)]{kol}
 { Kolmogorov A. N., 1940.} { Wienershe spiralen and einige andere interessante kurven im Hilbertschen raum, C. R. (Doklady) Acad. Sci. URSS (NS)}.  \textbf{26}, 115-118

\bibitem [Nualart (2006)] {nualart}
\textsc{ Nualart  D., 2006.}   { The Malliavin Calculus and
Related Topics, second edition}.  Springer-Verlag, Berlin.

 \bibitem [Ren and Sun (2009)]{ren}
  {Ren, Y., Sun, D., 2009.}  {Second-order neutral impulsive stochastic
differential equations with delay.}  J. Math. Phys. 50, 102709.

\bibitem [Pazy  (1983)] {pazy}
 \textsc{ Pazy A., 1983. }    {Semigroups of Linear Operators and Applications to Partial Differential Equations}.
   Applied Mathematical Sciences, vol. 44, Springer-Verlag, New York.

   \bibitem [Mandelbrot and Van Ness (1968)]{MV}
   {B. B. Mandelbrot, and J. W. Van Ness 1968.}  {Fractional Brownian motions, fractional noises and applications, SIAM Rev.}  \textbf{10} (4), 422-437.


\end{thebibliography}
\end{document}